\documentclass[10pt, a4paper]{amsart}
\usepackage{amsmath, amssymb}
\usepackage{color, ulem}

\numberwithin{equation}{section}
\newtheorem{theorem}{Theorem}[section]

\newtheorem{proposition}[theorem]{Proposition}

\theoremstyle{remark}
\newtheorem{remark}{Remark}[section]

\theoremstyle{definition}

\newtheorem{example}[theorem]{Example}

\newcommand{\HH}{{\mathcal{H}}}

\newcommand{\R}{{\mathbb{R}}}
\newcommand{\C}{{\mathbb{C}}}
\newcommand{\Z}{{\mathbb{Z}}}

\begin{document}

\title%
[Weak dispersion for the magnetic Dirac flow]%
{Virial identity and weak dispersion for the magnetic Dirac
equation}
\begin{abstract}
  We analyze the dispersive properties of a Dirac system
  perturbed with a magnetic field. We prove a
  general virial identity; as applications, we obtain smoothing
  and endpoint Strichartz estimates which are optimal from the decay
  point of view. We also prove a Hardy-type inequality for the
  perturbed Dirac operator.
\end{abstract}
\date{\today}    
\author{Nabile Boussaid}
\address{Nabile Boussaid: D\'epartement de Math\'ematiques,
Universit\'e de Franche-Comt\'e,
16 route de Gray,
25030 Besan\c{c}on Cedex,
France}
\email{nabile.bousssaid@univ-fcomte.fr}

\author{Piero D'Ancona}
\address{Piero D'Ancona:
SAPIENZA - Universit\`a di Roma, Dipartimento di Matematica,
Piaz\-zale A.~Moro 2, I-00185 Roma, Italy}
\email{dancona@mat.uniroma1.it}

\author{Luca Fanelli}
\address{Luca Fanelli:
Universidad del Pais Vasco, Departamento de
Matem$\acute{\text{a}}$ticas, Apartado 644, 48080, Bilbao, Spain}
\email{luca.fanelli@ehu.es}

\subjclass[2000]{
35L05, 
35Q40, 
58J50, 
}
\keywords{%
Dirac equation,
smoothing estimates,
Strichartz estimates
wave equation,
dispersive equations,
magnetic potential }

\maketitle


\section{Introduction}\label{sec:intro} 

The Dirac equation is one of the fundamental systems
of modern physics and mathematics, used to describe a spin 1/2
particle in quantum electrodynamics. Even if the physical
interpretation of Dirac fields is not completely unambiguous,
the very rich mathematical structure of this system
makes it an interesting
object of study. We refer to \cite{T} for a thorough treatment of
the subject, including the physical validity of the
model.

We fix our notations.
The Dirac equation is the $4\times4$ constant coefficient system
\begin{equation}\label{eq:freedir}
  iu_{t}=m \beta u+\mathcal{D} u,\qquad m\in \mathbb{R}
\end{equation}
where $u:\mathbb{R}_{t}\times
\mathbb{R}^{3}_{x}\to \mathbb{C}^{4}$, the operator $\mathcal{D}$
is defined as
\begin{equation*}
  \mathcal{D}=i^{-1}\sum_{k=1}^{3}\alpha_{k}\partial_{k}
\end{equation*}
and the $4\times4$ \emph{Dirac matrices} can be written
\begin{equation*}
  \alpha_{k}=
  \begin{pmatrix}
  0 &\sigma_{k}  \\
  \sigma_{k} & 0
  \end{pmatrix},\quad
  \beta=
  \begin{pmatrix}
  I_{2} & 0\\
   0& -I_{2}
  \end{pmatrix},\qquad
  k=1,2,3
\end{equation*}
in terms of the \emph{Pauli matrices}
\begin{equation*}
  I_{2}=
  \begin{pmatrix}
  1 &0 \\
  0 &1
  \end{pmatrix},\quad
  \sigma_{1}=
  \begin{pmatrix}
  0 &1 \\
  1 &0
  \end{pmatrix},\quad
  \sigma_{2}=
  \begin{pmatrix}
  0 &-i \\
  i &0
  \end{pmatrix},\quad
  \sigma_{3}=
  \begin{pmatrix}
  1 &0 \\
  0 &-1
  \end{pmatrix}.
\end{equation*}
The coefficient $m$ is called the \emph{mass} and we shall
distingush the \emph{massless} case $m=0$ from the
\emph{massive} case $m\neq0$. This
distinction is important in relation to
the dispersive properties of the equation.
Indeed, the commutation rules
\begin{equation*}
\alpha_{\ell}\alpha_{k}+\alpha_{k}\alpha_{\ell}
=2\delta_{kl}I_{4}
\end{equation*}
imply the identity
\begin{equation*}
     \mathcal{D}^{2}=-\Delta I_{4}.
\end{equation*}
Thus we have the property
\begin{equation*}
  (i\partial_{t}-\mathcal{D})(i\partial_{t}+\mathcal{D})=
  (\Delta-\partial^{2}_{tt}) I_{4}
\end{equation*}
showing the intimate relation between the Dirac and the wave equation.
A similar computation in the massive case produces a Klein-Gordon
equation with positive mass $m^{2}$:
\begin{equation*}
  (i\partial_{t}-\mathcal{D}-m \beta)(i\partial_{t}+\mathcal{D}+m \beta)=
  (\Delta-m^{2}-\partial^{2}_{tt}) I_{4}
\end{equation*}
It is then straightforward to derive dispersive and smoothing properties
of the free flows from the corresponding ones for the scalar equations
using these identities (see e.g. \cite{pda-lf2}).
Our purpose here is to extend these
properties to the case of system perturbed by a
magnetostatic potential
\begin{equation*}
  A=A(x)=(A^1(x),A^2(x),A^3(x)):\mathbb{R}^{3}\to \mathbb{R}^{3}.
\end{equation*}
In virtue of the principle of minimal electromagnetic coupling, the influence of the field is introduced in the equation by replacing
the standard derivatives with the \emph{covariant} derivatives
\begin{equation*}
  \nabla_A:=\nabla-iA,
\end{equation*}
so that the operator $\mathcal{D}$ is replaced by
\begin{equation}\label{eq:diracopmassless}
  \mathcal{D}_{A}=i^{-1}\sum_{j=1}^{3}\alpha_{k}(\partial_{k}-i A^{k})
\end{equation}
We shall also use the unified notation
\begin{equation}\label{eq:diracop}
  \HH=i^{-1}\alpha\cdot\nabla_{A}+m\beta=
    \mathcal{D}_{A}+m \beta
\end{equation}
for the perturbed operator, which covers both the massive and
the massless case. Thus our main goal here is to
investigate the dispersive properties of the flow
$u=e^{it\HH}f$ relative to the Cauchy problem
\begin{equation}\label{eq:dirac}
    iu_t(t,x)+\HH u(t,x)=0, \qquad
    u(0,x)=f(x)
\end{equation}
where $u(t,x):\R\times\R^3\to\C^4$, $f(x):\R^3\to\C^4$, and $m\in\R$.

It is natural to require that the operator $\HH$ be selfadjoint.
Several sufficient conditions on the potential $A$ are
known; e.g., if the field $A$ is smooth or satisfies
\begin{equation*}
  |A(x)|\le \frac{a}{|x|}+b,\qquad
  a<1,\quad b>0
\end{equation*}
then $\HH$ admits a unique selfadjoint extension. We refer to
\cite{T} for a discussion of this problem; here we prefer to
make an all-encompassing abstract assumption on the operator:

\begin{quote}
  \textsc{Self-adjointness assumption (A)}: the operator $\HH$ is
  essentially selfadjoint on $C^{\infty}_{c}(\mathbb{R}^{n})$,
  and in addition, for initial data in $C^{\infty}_{c}(\mathbb{R}^{n})$,
  the flow $e^{it\HH}f$ belongs to $C(\mathbb{R},H^{3/2})$.
\end{quote}
The density condition allows to approximate
rough solutions with smoother ones, locally uniformly in time,
and is easily verified in concrete cases. 

Dispersive, smoothing and Strichartz estimates for a
perturbed Dirac equation of the form
\begin{equation*}
  iu_{t}=\HH_{0}u+V(x)u
\end{equation*}
were obtained earlier in \cite{pda-lf1, pda-lf2, B, B2}, for a general potential
$V=V^{*}\in \mathbb{C}^{4 \times4}$ satisfying suitable smallness
and decay conditions. In those works we used a perturbative
approach, relying heavily on spectral methods.
Here we follow a different
approach, based on multiplier methods, with two major
advantages. First, we can partially overcome the smallness
assumption; and second, the assumptions are more natural from
the physical point of view since they are expressed in terms
of the \emph{magnetic field}
\begin{equation*}
  B=\mathop\mathrm{curl}A,
\end{equation*}
which is the physically relevant quantity.
Actually, all assumptions are in terms of the quantities
\begin{equation*}
  B_{\tau}=\frac{x}{|x|}\wedge B,
  \qquad
  \partial_rB=(\partial_rB^1,\partial_rB^2,\partial_rB^3),
\end{equation*}
which are, respectively, the \emph{tangential component} and
the \emph{radial derivative} of the field $B$.
Moreover, we establish for the first time a \emph{virial identity}
for the perturbed Dirac equation, which has several application
not restricted to smoothing properties of the solution.
Multiplier methods in relation with weak dispersion
properties have a long story, starting from Morawetz
\cite{mor} for the Klein-Gordon equation and
\cite{cs}, \cite{s}, \cite{v} for the Schr\"odinger equation,
and adapted to  more general situations in \cite{pv}, \cite{pv2}.
Potential perturbations for the Schr\"odinger equation were
considered in \cite{brv}, \cite{brvv}, while the magnetic case
was studied in \cite{FV}. The perturbed Dirac equation was
studied in \cite{pda-lf1}, \cite{pda-lf2}, \cite{B} and \cite{B2} using spectral
instead of multiplier methods, which are applied for the first
time here.

The paper is organized as follows. In the rest of the Introduction
we shall describe the main results of the paper, namely a
general virial identity and optimal smoothing and Strichartz
estimates for the Dirac equation perturbed with a magnetostatic
potential. Sections 2, 3 and 4 are devoted to the proofs. In the
Appendix we prove a magnetic Hardy inequality (Proposition
\ref{thm:hardy}) which is elementary but has maybe an independent
interest.

\subsection{Virial identities}

The Dirac operator does not have a definite sign and this is a
substantial difficulty for a direct application of multiplier methods.
To overcome it we shall resort to the \emph{squared} Dirac equation
\begin{equation*}
  (i\partial_t-\HH)(i\partial_t+\HH)=(-\partial_{tt}-\HH^2).
\end{equation*}
Thus we are reduced
to study a diagonal system of wave (Klein-Gordon) equations of the form
\begin{equation}\label{eq:wave}
    u_{tt}(t,x)+Lu(t,x)=0,\qquad
    L=(m^{2}-\Delta)I
\end{equation}
with $u=u(t,x):\R\times\R^3\to\C^4$.
Our first result is a formal virial identity
for solutions of a general system of wave equations like
\eqref{eq:wave}, with $L$ being any selfadjoint operator on
$L^{2}(\mathbb{R}^{n};C^{k})$. In the following, round brackets
\begin{equation*}
  (F,G)=\int_{\mathbb{R}^{n}}F \cdot\overline{G}dx
\end{equation*}
denote the inner product in $L^2(\R^n;\C^k)$,
while $[S,T]=ST-TS$ is the commutator of operators.

\begin{theorem}[Virial identity for the wave equation]
                 \label{thm:virial1}
  Given a function $\phi:\mathbb{R}^{n}\to \mathbb{R}$, define
  the quantity
  \begin{equation}\label{eq:teta}
    \Theta(t)=\left(\phi u_t,u_t\right)+\Re\left(\left(2\phi
    L-L\phi\right)u,u\right).
  \end{equation}
  Then any solution $u(t,x)$ of
  \eqref{eq:wave} satisfies the formal identities
  \begin{align}
    & \dot\Theta(t)=\Re\left([L,\phi]u,u_t\right)
    \label{eq:tetadot}
    \\
    & \ddot\Theta(t)=-\frac12\left([L,[L,\phi]]u,u\right).
    \label{eq:tetadotdot}
  \end{align}
\end{theorem}

We recall that the virial identity for a magnetic wave equation
first appeared in \cite{F,FV}; the
abstract formulation given here can be applied to more general
equations including magnetic Dirac systems.
Thus, as an application of \eqref{eq:tetadot},
\eqref{eq:tetadotdot} we obtain:

\begin{theorem}[Virial identity for the Dirac equation]
      \label{thm:virial2}
  Assume the operator $\HH$ defined by \eqref{eq:diracop}
  satisfies assumption (A) and
  let $\phi:\mathbb{R}^{3}\to\R$ be a real valued function.
  Then any solution $u(t,x)$ of \eqref{eq:dirac}
  satisfies the formal virial identity,
  where $DB=[\partial_{j}B_{i}]_{i,j=1,3}$ and $S=\frac i4\alpha\wedge\alpha$ is the spin operator:
  \begin{align}\label{eq:virialdirac}
    2\int_{\R^3} & \nabla_AuD^2\phi\overline{\nabla_A u}
    -\frac12\int_{\R^3}|u|^2\Delta^2\phi
    +2\Im\int_{\R^3}u\phi'B_{\tau}\cdot\overline{\nabla_A u}
    +2\int_{\R^3}\left[S\cdot\left(
    DB\nabla\phi\right)u\right]\cdot\overline u
    \\
    & = -\frac{d}{dt}\Re\left(\int_{\R^3}u_t(2\nabla\phi\cdot\overline{\nabla_A u}
    +\overline u\Delta\phi)\right).
    \nonumber
  \end{align}
\end{theorem}

\begin{remark}\label{rem:radial}
  In the following, we will always consider radial multipliers $\phi$,
  in which case the last term at the left-hand side of
  \eqref{eq:virialdirac} simplifies to
  \begin{equation*}
    S\cdot DB\nabla\phi = \phi'S\cdot\partial_rB.
  \end{equation*}
\end{remark}

In order to deduce a smoothing estimate from the virial identity,
it will be necessary to impose suitable smallness conditions on the
components of $B$ appearing in \eqref{eq:virialdirac}, for which no natural
positivity assumption holds in general.

\subsection{Local smoothing for the magnetic Dirac Equation}

As a first application of the virial identity, we prove some
smoothing estimates for the magnetic Dirac equation. For $f:\R^3\to\C$,
denote by
\begin{equation*}  \|f\|_{L^p_rL^\infty(S_r)}:=\|\sup_{|x|=r}|f|\,\|_{L^p_r}=\left(\int_0^{+\infty}
  (\sup_{|x|=r}|f|)^p dr\right)^{\frac1p}.
\end{equation*}
Moreover, by $\nabla_A^r u$ and $\nabla_A^\tau u$ we denote,
respectively, the radial and tangential components of the covariant
gradient $\nabla_{A}=\nabla-iA(t,x)$:
\begin{equation}\label{eq:radtan}
  \nabla^{r}_{A}u=\frac{x}{|x|}\cdot \nabla_{A} u,\qquad
  \nabla^{\tau}_{A}u=\nabla_{A}u-\frac{x}{|x|}\nabla^{r}_{A}u,
\end{equation}
so that
\begin{equation*}
  |\nabla_A^r u|^2+|\nabla_A^\tau u|^2=|\nabla_Au|^2.
\end{equation*}
We can now state our main result.

\begin{theorem}\label{thm:smoothing}
  Let $\HH$ satisfy the self-adjointness assumption (A). Let
  $B=\mathop\mathrm{curl}A = B_{1}+B_{2}$
  with $B_{2}\in L^{\infty}(\mathbb{R}^{n})$
  and inroduce the quantities
  \begin{equation*}
    C_{0}=
    \||x|^{2}B_{1}\|_{L^{\infty}(\mathbb{R}^{n})},\qquad
    C_{1}=
    \||x|^\frac32B_{\tau}\|_{L^2_rL^\infty(S_r)},
    \qquad
    C_{2}=
    \||x|^2\partial_rB\|_{L^1_rL^\infty(S_r)}.
  \end{equation*}
  We shall assume the smallness conditions
  \begin{equation}\label{eq:condition}
    C_{0}<\frac14,\qquad{}
    C_{1}^{2}+3C_{2}+C_{1}\sqrt{C_{1}^{2}+6C_{2}}\le1
  \end{equation}
  and that the $L^{\infty}$ part of $B$ is absent in the massless case:
  \begin{equation}\label{eq:condition2}
    m=0 \quad \implies \quad B_{2}\equiv0.
  \end{equation}
  Then for all $f\in L^{2}$, the following estimate holds:
  \begin{equation}\label{eq:smoodirac}
    \sup_{R>0}\frac1R\int_{-\infty}^{+\infty}\int_{|x|\leq
    R}|e^{it\HH}f|^2dxdt\lesssim\|f\|_{L^2}^2.
  \end{equation}
  Assume moreover that
  the second inequality in \eqref{eq:condition} is strict;
  then for any
  $f\in D(\HH)$ the following estimate is true:
  \begin{align}
 \sup_{R>0}\frac1R\int_{-\infty}^{+\infty}&\int_{|x|\leq
    R}  |\nabla_A e^{it\HH}f|^2dxdt+
    \|e^{it\HH}f\|_{L^\infty_xL^2_t}^2
    \label{eq:smoodirac2}
    \\
    +\int_{-\infty}^{+\infty}\int_{\R^3} & \frac{|\nabla_A^\tau e^{it\HH}f|^2}{|x|}dxdt
    +\sup_{R>0}\frac1{R^2}\int_{-\infty}^{+\infty}\int_{|x|=R}|e^{it\HH}f|^2d\sigma dt
    \lesssim\|\HH f\|_{L^2}^2.
    \nonumber
  \end{align}
\end{theorem}
\begin{remark}\label{rem:gauge}
  Notice that all the assumptions in the smoothing Theorem \ref{thm:smoothing} are expressed in terms of the magnetic field $B$, and consequently the gauge invariance of the result is preserved.
\end{remark}
\begin{example}\label{ex:2}
  Explicit examples of magnetic fields satisfying assumption
  \eqref{eq:condition} are of the following form
  \begin{equation}\label{eq:perturb}
    \omega\left(\frac{x}{|x|}\right)\frac{x}{|x|}
    +\epsilon B(x),
  \end{equation}
  where $\omega$ is a smooth function on the unit sphere, while
   $\epsilon$ is sufficiently small and $B:\R^3\to\R^3$ satisfies
  \begin{equation*}
    |B_\tau(x)|\leq\frac{1}{|x|^{2-\delta}+|x|^{2+\delta}},
    \qquad
    |\partial_rB|\leq\frac{1}{|x|^{3-\delta}+|x|^{3+\delta}},
  \end{equation*}
  for some $\delta>0$.
\end{example}

In the proof of Theorem \ref{thm:smoothing}, we shall use
the following Hardy inequality for the magnetic Dirac operator,
proved in the Appendix at the end of the paper.
Compare with \cite{DELV} for parallel results.

\begin{proposition}\label{thm:hardy}
  Let $B=\mathop \mathrm{curl}A=B_{1}+B_{2}$ and assume that
  \begin{equation}\label{eq:condC0}
    C_{0}=
    \||x|^{2}B_{1}\|_{L^{\infty}(\mathbb{R}^{n})}<\infty,\qquad
    \|B_{2}\|_{L^{\infty}(\mathbb{R}^{n})}<\infty
  \end{equation}
  Then, for any $f:\R^3\to\C^4$ such that $\mathcal \HH f\in L^2$,
  and any $\epsilon<1$,
  the following inequality holds:
  \begin{equation}\label{eq:hardy}
    m^{2}\int|f|^{2}+
    \left(\frac{1-\epsilon}4-C_{0}\right)
    \int\frac{|f|^2}{|x|^2}+
    \epsilon\int|\nabla_{A}f|^{2}
    \leq
    \left(1+\frac{\|B_{2}\|_{L^{\infty}}}{m^{2}}\right)
    \int_{\R^3}|\mathcal
    \HH f|^2.
  \end{equation}
\end{proposition}

\noindent
When $m=0$, $B_{2}\equiv0$, the right hand side is to be intepreted
simply as $\int|\HH f|^{2}dx$.

\subsection{Strichartz Estimates}\label{subsec:stri}

As a natural application of the local smoothing estimates,
we now derive from them the Strichartz estimates for the perturbed
Dirac equation. We recall that the solution
$u(t,x)=e^{it \mathcal{D}}f$ of the free massless Dirac system
with initial value $u(0,x)=f(x)$ satisfies
\begin{equation}\label{eq:strhom}
  \|e^{it \mathcal{D}}f\|
  _{L^{p}\dot H_{q}^{\frac 1q-\frac1p-\frac12}}
  \lesssim\|f\|_{L^{2}},
\end{equation}
for all wave admissible $(p,q)$
\begin{equation}\label{eq:wadm}
  \frac2p+\frac{2}{q}=\frac{2}{2},\qquad
  2< p\le \infty,\qquad
  \infty> q\ge 2
\end{equation}
while in the massive case $m\neq0$ we have
\begin{equation}\label{eq:strhommass}
  \|e^{it(\mathcal{D}+m\beta)}f\|
       _{L^{p} H_{q}^{\frac 1q-\frac1p-\frac12}}
  \lesssim\|f\|_{L^{2}},
\end{equation}
for all Schr\"odinger admissible $(p,q)$
\begin{equation}\label{eq:sadm}
  \frac2p+\frac 3q=\frac 32,\qquad
  2\le p\le \infty,\qquad
  6\ge q\ge 2.
\end{equation}
For a proof of these estimates see \cite{pda-lf2}.
In the perturbed case we obtain exactly the same results:

\begin{theorem}\label{thm:stri}
  Assume $\HH$ and $\mathcal D_A$ satisfy (A).
  Moreover, assume that \eqref{eq:condition}, \eqref{eq:condition2}
  hold and that
  \begin{equation}\label{eq:decay}
    \sum_{j\in\Z}2^j\sup_{|x|\sim2^j}|A|<\infty.
  \end{equation}
  Then the massless perturbed flow satisfies the Strichartz estimates
  \begin{equation}\label{eq:stridir2}
    \|e^{it\mathcal D_A}f\|_{L^p\dot H^{\frac1q-\frac1p-\frac12}_q}
    \lesssim\|f\|_{L^{2}}
  \end{equation}
  for all wave admissible couple $(p,q)$,
  (in particular, $p\neq2$), while in the massive case we have,
  for all Schr\"odinger admissible couple $(p,q)$,
  \begin{equation}\label{eq:stridir3}
    \|e^{it\HH}f\|_{L^p H^{\frac1q-\frac1p-\frac12}_q}
    \lesssim\|f\|_{L^{2}} \qquad
    (m\neq0).
  \end{equation}
\end{theorem}

\section{Proof of the Virial Identities}

\begin{proof}[Proof of Theorem \ref{thm:virial1}]
  The proof relies on a direct computation. By
  equation \eqref{eq:wave} we have
  \begin{equation}\label{eq:0}
    \frac{d}{dt}(\phi u_t,u_t)=-2\Re(\phi Lu,u_t),
  \end{equation}
  and
  \begin{equation}\label{eq:00}
    \frac{d}{dt}\Re((2\phi L-L\phi)u,u)=\Re((\phi L+L\phi)u,u_t),
  \end{equation}
  since $\phi$ and $L$ are symmetric operators. Summing
  \eqref{eq:0} and \eqref{eq:00} we get \eqref{eq:tetadot}.
  An additional differentiation gives
  \begin{equation}\label{eq:000}
    \frac{d}{dt}\Re([L,\phi]u,u_t)=\Re([L,\phi]u_t,u_t)-\Re([L,\phi]u,Lu).
  \end{equation}
  Since $[L,\phi]$ is anti-symmetric, we have
  \begin{equation}\label{eq:0000}
    \Re([L,\phi]u_t,u_t)=0
  \end{equation}
  and also
  \begin{equation}\label{eq:00000}
    -\Re([L,\phi]u,Lu)=-\frac12\left\{([L,\phi]u,Lu)+(Lu,[L,\phi]u)\right\}
    =-\frac12([L,[L,\phi]]u,u).
  \end{equation}
  Identities \eqref{eq:000}, \eqref{eq:0000} and \eqref{eq:00000}
  give \eqref{eq:tetadotdot}.
\end{proof}


\begin{proof}[Proof of Theorem \ref{thm:virial2}]
  Let $u$ be a solution to equation \eqref{eq:dirac}.
  Using the identity
  \begin{equation*}
    0=(i\partial_t-\HH)(i\partial_t+\HH)u=(-\partial_{tt}-\HH^2)u,
  \end{equation*}
  we see that $u$ solves a Cauchy problem of the form \eqref{eq:wave}:
  \begin{equation}\label{eq:diracwave}
    \begin{cases}
      u_{tt}+\HH^2u=0
      \\
      u(0)=f
      \\
      u_t(0)=i\HH f
    \end{cases}
    \qquad
    L=\HH^2.
  \end{equation}
  In order to apply
  \eqref{eq:tetadot}, \eqref{eq:tetadotdot}, we need to compute
  explicitly the commutators appearing in the formulas
  with the choice $L=\HH^2$.

  In the following we shall need the spin operator $S$, defined as
  the triplet of matrices
  \begin{equation*}
    S=\frac i4\alpha\wedge\alpha=
      \frac i4(\alpha_{2}\alpha_{3}-\alpha_{3}\alpha_{2},\
               \alpha_{3}\alpha_{1}-\alpha_{1}\alpha_{3},\
               \alpha_{1}\alpha_{2}-\alpha_{2}\alpha_{1}).
  \end{equation*}
  We also recall the formula
  \begin{equation}\label{eq:spin}
    (\alpha\cdot F)(\alpha\cdot G)=F\cdot G+2iS\cdot(F\wedge G)
  \end{equation}
  which holds for any matrix-valued vector fields $F=(F^1,F^2,F^3)$,
  $G=(G^1,G^2,G^3)$,  with
  $F^i,G^i\in\mathcal M_{4\times 4}(\C)$ (see \cite{T} for
  an extensive list of algebraic identities connected to
  Dirac operators). Thus expanding the square $\HH^2$ we have
  \begin{equation}\label{eq:pert}
    \HH^2=\HH_0^2-\HH_0(\alpha\cdot A)-(\alpha\cdot A)
    \HH_0+(\alpha\cdot A)(\alpha\cdot A),
  \end{equation}
  where the unperturbed part is precisely
  \begin{equation}\label{eq:pert1}
    \HH_0^2=(m^{2}-\Delta)I_{4},
  \end{equation}
  and $I_4$ denotes the identity matrix. Using \eqref{eq:spin} we
  compute
  \begin{equation}\label{eq:pert2}
    -\HH_0(\alpha\cdot A)-(\alpha\cdot A)\HH_0+(\alpha\cdot A)(\alpha\cdot
    A)=i(\nabla\cdot A)+i(A\cdot\nabla)+|A|^2-2S\cdot\left(\nabla\wedge
    A+A\wedge\nabla\right).
  \end{equation}
  Now observe that
  \begin{align}
    \nabla\wedge
    A+A\wedge\nabla & =\text{curl}A=B
    \label{eq:pert3}
    \\
    -\Delta+i(\nabla\cdot A)+i(A\cdot\nabla)+|A|^2 & =
    (i \nabla-A)^{2}=
    -\Delta_A.
    \label{eq:pert4}
  \end{align}
  In conclusion, by \eqref{eq:pert}, \eqref{eq:pert1},
  \eqref{eq:pert2}, \eqref{eq:pert3}, \eqref{eq:pert4} we obtain
  \begin{equation}\label{eq:square}
    \HH^2=(m^{2}-\Delta_A)I_4-2S\cdot B.
  \end{equation}
   Analogously, in the massless case we have
  \begin{equation}\label{eq:squareless}
     \mathcal D_A^2=-\Delta_AI_4-2S\cdot B.
 \end{equation}
  Hence the commutator with $\phi$ reduces to
  \begin{equation}\label{eq:comm1a}
    [\HH^2,\phi]=[m^{2},\phi]-[\Delta_A,\phi]-2[S\cdot B,\phi]=
    -[\Delta_A,\phi].
  \end{equation}
  Using the Leibnitz rule
  \begin{equation*}
    \nabla_A(fg)=g\nabla_Af+f\nabla g,
  \end{equation*}
  we arrive at the explicit formula
  \begin{equation}\label{eq:comm1}
    [\HH^2,\phi]=-[\Delta_A,\phi]=-2\nabla\phi\cdot\nabla_A-(\Delta\phi).
  \end{equation}
  Recalling \eqref{eq:teta}, \eqref{eq:tetadot}, we obtain
  \begin{equation}\label{eq:tetadotdirac}
    \dot\Theta(t)=-\Re\left(\int_{\R^3}
      u_t(2\nabla\phi\cdot\overline{\nabla_A u}
      +\overline u\Delta\phi)\right)
  \end{equation}

  We now compute the second commutator between $\HH^2$ and $\phi$.
  By \eqref{eq:square}, \eqref{eq:comm1} we have
  \begin{equation}\label{eq:comm10}
    [\HH^2,[\HH^2,\phi]]=[\Delta_A,[\Delta_A,\phi]]+2[S\cdot
    B,[\Delta_A,\phi]].
  \end{equation}
  The term involving the magnetic Laplacian gives
  \begin{align}
  (u,[\Delta_A,[\Delta_A,\phi]]u) = &
  4\int_{\R^n}\nabla_A uD^2\phi\overline{\nabla_A u}
  -\int_{\R^n}|u|^2\Delta^2\phi
  \label{eq:comm11}
  \\
  &
  +4\Im\int_{\R^n}u\phi'B_\tau\cdot\overline{\nabla_{A}u}
  \nonumber
\end{align}
(see formula (2.18) in \cite{FV} with $V\equiv0$).
By \eqref{eq:comm1}, the last term in \eqref{eq:comm10} is equal to
\begin{equation*}
    [S\cdot B,[\Delta_A,\phi]]=2[S\cdot B,\nabla\phi\cdot\nabla_A]
    =2(S\cdot B\nabla\phi\cdot\nabla_A-\nabla\phi\cdot\nabla_AS\cdot
    B).
\end{equation*}
Both $\phi$ and the components of the field $B$ are scalars;
moreover, we have
\begin{equation*}
  [B,\nabla_A]=-DB,
\end{equation*}
where $DB$ denotes the (matrix) gradient of the field $B$,
and in conclusion
\begin{equation}\label{eq:comm12}
  2[S\cdot B,[\Delta_A,\phi]]=-4S\cdot\left(DB\nabla\phi \right).
\end{equation}
Finally,
identity \eqref{eq:virialdirac} follows from \eqref{eq:tetadotdot},
\eqref{eq:tetadotdirac}, \eqref{eq:comm10}, \eqref{eq:comm11} and
\eqref{eq:comm12}.
\end{proof}

\section{The smoothing estimates}\label{sec.smoo}

The formal computations leading to the
virial identity \eqref{eq:virialdirac} make
sense for sufficiently smooth solutions $u\in C(\mathbb{R},H^{3/2})$ and the choice of multiplier $\phi$ we make below.
Thanks to the density
assumption (A), if we approximate data $f\in L^{2}$ (resp. $D(\HH)$)
with $f_{j}\in C^{\infty}_{c}$,
the corresponding solutions
$u_{j}=e^{it\HH}f_{j}$ will converge to the solution $u=e^{it\HH}f$ in $C([-T,T];L^{2})$
(resp. $C([-T,T];D(\HH))$) for all $T>0$.

We shall apply identity \eqref{eq:virialdirac} to the solution
$u=e^{it\HH}f$ of the problem
  \begin{equation}\label{eq:dirac2}
    \begin{cases}
      iu_t=\HH u
      \\
      u(0)=f
    \end{cases}
  \end{equation}
with an appropriate multiplier function $\phi$.

\subsection{Choice of the multiplier}

Writing $r=|x|$, we define $\phi$ as follows (see \cite{FV})
  \begin{equation*}
  \phi_0(x)=\int_0^{r=|x|}\phi_0'(s)\,ds,
\end{equation*}
where
\begin{equation*}
  \phi'_0=\phi'_0(r)=
  \begin{cases}
    M+\frac13r,
    \qquad
    r\leq1
    \\
    M+\frac12-\frac1{6r^2},
    \qquad
    r>1,
  \end{cases}
\end{equation*}
and $M$ is a positive constant we will choose later. We have
\begin{equation*}
  \phi_0''(r)=
  \begin{cases}
    \frac13,
    \qquad
    r\leq1
    \\
    \frac1{3r^3},
    \qquad
    r>1
  \end{cases}
\end{equation*}
while the bilaplacian is given by
\begin{equation*}
  \Delta^2\phi_0(r)=-4\pi\delta_{x=0}-\delta_{|x|=1}.
\end{equation*}
Moreover, for any $R>0$ we define
\begin{equation*}
  \phi_R(r)=R\phi_0\left(\frac rR\right),
\end{equation*}
so that by rescaling we have
\begin{equation}\label{eq:fi1}
  \phi'_R(r)=
  \begin{cases}
    M+\frac r{3R},
    \qquad
    r\leq R
    \\
    M+\frac12-\frac{R^2}{6r^2},
    \qquad
    r>R
  \end{cases}
\end{equation}
\begin{equation}\label{eq:fi2}
  \phi''_R(r)=
  \begin{cases}
    \frac1{3R},
    \qquad
    r\leq R
    \\
    \frac1R\cdot\frac{R^3}{3r^3},
    \qquad
    r>R
  \end{cases}
\end{equation}
\begin{equation}\label{eq:filapl}
  \Delta\phi_R(r)=
  \begin{cases}
    \frac1R+\frac{2M}{r},
    \qquad
    r\leq R
    \\
    \frac{1+2M}{r},
    \qquad
    r>R
  \end{cases}
\end{equation}
\begin{equation}\label{eq:fibi}
  \Delta^2\phi_R(r)=-4\pi\delta_{x=0}-\frac1{R^2}\delta_{|x|=R}.
\end{equation}
We notice that $\phi'_R,\phi''_R,\Delta\phi_R\geq0$ and moreover
\begin{equation}\label{eq:estimatephi}
  \sup_{r\geq0}\phi'_R(r)\leq M+\frac12
  \qquad
  \sup_{r\geq0}\phi''_R(r)\leq \frac{1}{3R},
  \qquad
  \Delta\phi_R(r)\leq\frac{1+2M}{r}.
\end{equation}

\subsection{Estimate of the RHS in \eqref{eq:virialdirac}}

Consider the expression
\begin{equation*}
  \int_{\R^3}u_t(2\nabla\phi\cdot\overline{\nabla_A u}
    +u\Delta\phi)=
    (u_{t},2\nabla\phi\cdot{\nabla_A u}
      +\overline u\Delta\phi)_{L^{2}}
\end{equation*}
appearing at the right hand side in \eqref{eq:virialdirac}.
Using the equation, we can replace $u_{t}$ with
\begin{equation*}
  u_{t}=-i\HH u=-im \beta u -i \mathcal{D}_{A}u.
\end{equation*}
By the selfadjointess of $\beta$, it is easy to check that
\begin{equation*}
  \Re\left[
  -im(\beta u,2\nabla\phi\cdot\nabla_A u)
  -im(\beta u,\Delta \phi u)\right]=0,
\end{equation*}
so that
\begin{equation*}
  \Re
  (u_{t},2\nabla\phi\cdot{\nabla_A u}+u \Delta \phi)=
    2\Im(\mathcal{D}_A u,\nabla\phi\cdot{\nabla_A u})
      +\Im(\mathcal{D}_A u,\Delta \phi u)
\end{equation*}
and by Young we obtain
\begin{equation}\label{eq:boundary00}
  \left|\Re\left(\int_{\R^3}u_t(2\nabla\phi\cdot\overline{\nabla_A u}
    +\overline u\Delta\phi)\right)\right|
    \leq\frac32\|\mathcal{D}_{A}u\|_{L^2}^2+\|\nabla\phi\cdot\nabla_A u\|_{L^2}^2
    +\frac12\|u\Delta\phi\|_{L^2}^2.
\end{equation}
Recalling \eqref{eq:estimatephi},
and using Proposition \ref{thm:hardy} with the choice
$\epsilon =1-4C_0 $, which is positive in virtue of the
assumption $C_{0}<4^{-1}$, we have
\begin{equation}\label{eq:utgrad}
  \|\nabla\phi\cdot\nabla_A u\|_{L^2}^2
  \leq
  \frac{1}{1-4 C_{0}}
  \left(M+\frac12\right)\|\mathcal D_Au\|_{L^2}^2.
\end{equation}
The third term in \eqref{eq:boundary00}, can
be estimated using \eqref{eq:estimatephi} and
again the Hardy inequality
\eqref{eq:hardy}:
\begin{equation}\label{eq:ardi}
  \|u\Delta\phi\|_{L^2}^2\leq \frac{4}{1-4 C_{0}}(1+2M)\|\mathcal D_Au\|_{L^2}^2.
\end{equation}
Summing up, by \eqref{eq:boundary00}, \eqref{eq:utgrad}
and \eqref{eq:ardi} we conclude that
\begin{equation}\label{eq:boundary}
  \left|\Re\left(\int_{\R^3}u_t(2\nabla\phi\cdot\overline{\nabla_A u}
    +\overline u\Delta\phi)\right)\right|
    \lesssim
    \|\mathcal D_Au\|_{L^2}^2,
\end{equation}
for any $t\in\R$.

\subsection{Estimate of the LHS in \eqref{eq:virialdirac}}

In this subsection, we bound from below the time integral of the LHS in \eqref{eq:virialdirac}.

We shall use the elementary identity
\begin{equation}\label{eq:formula}
  \nabla_A uD^2\phi\overline{\nabla_Au}=
  \frac{\phi'(r)}{r}|\nabla_A^\tau u|^2+\phi''(r)|\nabla_A^ru|^2.
\end{equation}
By \eqref{eq:formula}, \eqref{eq:fi1}, \eqref{eq:fi2} and
\eqref{eq:fibi}, for the first two terms at the LHS of
\eqref{eq:virialdirac} we have
\begin{align}\label{eq:111}
    & 2\int_{\R^3}\nabla_AuD^2\phi_R\overline{\nabla_A u}
    -\frac12\int_{\R^3}|u|^2\Delta^2\phi_R
    \\
    & \geq\frac{2}{3R}\int_{|x|\leq R}|\nabla_A u|^2dx
    +2M\int_{\R^3}\frac{|\nabla_A^\tau u|^2}{|x|}dx
    +2\pi|u(t,0)|^2+\frac{1}{2R^2}\int_{|x|=R}|u|^2d\sigma(x),
    \nonumber
\end{align}
for any $R>0$, where $d\sigma$ denotes the surface measure on the
sphere of radius $R$. For the perturbative term involving $B_\tau$ in
\eqref{eq:virialdirac}, by the H\"older inequality and
\eqref{eq:estimatephi} we obtain
\begin{align}\label{eq:1111}
    & 2\Im\int_{t=-T}^{t=T}\int_{\R^3}u\phi'_RB_{\tau}\cdot\overline{\nabla_A u}dxdt
    \\
    & \geq -(2M+1)\left(\sup_{R>0}\frac1{R^2}\int_{t=-T}^{t=T}\int_{|x|=R}
    |u|^2d\sigma(x)dt\right)^{\frac12}\left(\int_{t=-T}^{t=T}\int_{\R^3}\frac{|\nabla_A^\tau u|^2}{|x|}dxdt\right)^{\frac12}
    \||x|^{\frac32}B_\tau\|_{L^2_rL^\infty(S_r)}.
    \nonumber
\end{align}
For the remaining term in \eqref{eq:virialdirac}, observe that the
operator norm of the components of $S=(S^1,S^2,S^3)$ is
\begin{equation*}
  \|S^k\|_{\C^4\to\C^4}=\frac12;
\end{equation*}
hence we can write
\begin{align}\label{eq:11111}
  & 2\int_{t=-T}^{t=T}\int_{\R^3}|u|^2S\cdot\left[\nabla\phi_RDB\right]dxdt
  \\
  & \geq -3\left(M+\frac12\right)\left(\sup_{R>0}\frac1{R^2}\int_{t=-T}^{t=T}\int_{|x|=R}
    |u|^2d\sigma_R(x)dt\right)\||x|^2\partial_rB\|_{L^1_rL^\infty(S_r)},
    \nonumber
\end{align}
since $\phi_R$ is radial. Now we introduce the norms
\begin{align*}
  \|u\|_{X,T}^2 & :=\sup_{R>0}\frac1R\int_{t=-T}^{t=T}\int_{|x|\leq R}|u|^2dxdt
  \\
  \|u\|_{Y,T}^2 & :=\sup_{R>0}\frac1{R^2}\int_{t=-T}^{t=T}\int_{|x|=R}|u|^2d\sigma(x)dt.
\end{align*}
Taking the supremum over $R>0$ in \eqref{eq:111} and summing with
\eqref{eq:1111}, \eqref{eq:11111}, we obtain
\begin{align}\label{eq:LHS}
  & 2\int_{t=-T}^{t=T}\int_{\R^3}\nabla_AuD^2\phi_R\overline{\nabla_A u}dt
    -\frac12\int_{t=-T}^{t=T}\int_{\R^3}|u|^2\Delta^2\phi_Rdt
    \\
  & +2\Im\int_{t=-T}^{t=T}\int_{\R^3}u\phi'_RB_{\tau}\cdot\overline{\nabla_A u}
    +2\int_{t=-T}^{t=T}\int_{\R^3}|u|^2S\cdot\left[\nabla\phi_RDB\right]dt
    \nonumber
    \\
  & \geq \frac23\|\nabla_Au\|_{X,T}^2+\left(\frac12-3\left(M+\frac12\right)
  \||x|^2\partial_rB\|_{L^1_rL^\infty(S_r)}\right)\|u\|_{Y,T}^2
    \nonumber
    \\
  & -(2M+1)\left(\int_{t=-T}^{t=T}\||x|^{-\frac12}\nabla_A^\tau
    u\|_{L^2}^2dt\right)^{\frac12}\||x|^{\frac32}B_\tau\|_{L^2_rL^\infty(S_r)}\|u\|_{Y,T}
    \nonumber
    \\
  & +2M\int_{t=-T}^{t=T}\||x|^{-\frac12}\nabla_A^\tau u\|_{L^2}^2dt+2\pi\int_{t=-T}^{t=T}|u(t,0)|^2dt.
  \nonumber
\end{align}
In order to deduce \eqref{eq:smoodirac},
\eqref{eq:smoodirac2}, we need to ensure the positivity of the
right-hand side of \eqref{eq:LHS}. Define $p,q$ as
\begin{equation*}
  p=\left(\int_{t=-T}^{t=T}\||x|^{-\frac12}\nabla_A^\tau
    u\|_{L^2}^2\,dt\right)^{\frac12},
    \qquad
  q=\|u\|_Y,
\end{equation*}
while $C_{1},C_{2}$ are defined in the statement of the Theorem.
Then we are led to study the inequality
\begin{equation}\label{eq:algebra}
  2Mp^2+\left(\frac12-3\left(M+\frac12\right)C_2\right)q^2-(2M+1)C_1pq\geq0
\end{equation}
and it is immediate to check that \eqref{eq:algebra} holds
for all $p,q\geq0$ and $M=C_1/(2\sqrt{C_1^2+6C_2})$, provided $C_1$ and $C_2$ satisfy
\eqref{eq:condition}. Thus, dropping the corresponding nonnegative
terms, we arrive at the estimate
\begin{align}\label{eq:LHSb}
 &2\int_{t=-T}^{t=T}\int_{\R^3}\nabla_AuD^2\phi_R\overline{\nabla_A u}dt
    -\frac12\int_{t=-T}^{t=T}\int_{\R^3}|u|^2\Delta^2\phi_Rdt
    \\
  & +2\Im\int_{t=-T}^{t=T}\int_{\R^3}u\phi'_RB_{\tau}\cdot\overline{\nabla_A u}
    +2\int_{t=-T}^{t=T}\int_{\R^3}|u|^2S\cdot\left[\nabla\phi_RDB\right]dt
    \nonumber
    \\
  & \geq \frac23\|\nabla_Au\|_{X,T}^2\ge
    \frac23\|\mathcal{D}_Au\|_{X,T}^2
  \nonumber
\end{align}
where in the last step we used the pointwise inequality
$|\mathcal{D}_Au|\le|\nabla_{A}u|$.
We now integrate in time the virial identity on $[-T,T]$, and using
\eqref{eq:LHSb} and \eqref{eq:boundary} we get
\begin{equation}\label{eq:intermed}
  \int_{-T}^{T}
  \|\mathcal D_Au\|_X^2 dt\lesssim
  \|\mathcal D_Au(T)\|_{L^2}^2+\|\mathcal D_Au(-T)\|_{L^2}^2.
\end{equation}
Notice that all the above computations do not depend on the
sign of the mass $m$.

Let us consider the range of $D_A$, form Proposition \ref{thm:hardy}, we have that for $C_0<1/4$, $0\not\in {\rm ker}(D_A)$ so ${\rm Ran}(D_A)$ is either $L^2$ if $0$ is not in the essential spectrum of $D_A$ or it is dense in $L^2$.
Now fix an arbitrary $g\in  {\rm Ran}(D_A)$, there exists  $f\in D(D_A)=D(\HH)$ with $\mathcal{D}_{A}f=g$; we consider then
the solution $u(t,x)$ of the problem
\begin{equation*}
  iu_{t}=-m \beta u+\mathcal{D}_{A}u,\qquad
  u(0,x)=f
\end{equation*}
with opposite mass,
and notice that $u$ satisfies estimate \eqref{eq:intermed}. If we
apply to this equation the operator $\mathcal{D}_{A}$ we obtain, by
the anticommutation rules,
\begin{equation*}
  i(\mathcal{D}_{A}u)_{t}=\beta m (\mathcal{D}_{A}u)+
   \mathcal{D}_{A}(\mathcal{D}_{A}u),\qquad
   (\mathcal{D}_{A}u(0,x))=\mathcal{D}_{A}f
\end{equation*}
or, in other words, the function $v=\mathcal{D}_{A}u$ solves the
problem
\begin{equation*}
  iv_{t}=\HH v,\qquad v(0,x)=g \quad \implies \quad
  v=e^{it\HH}g.
\end{equation*}
Hence \eqref{eq:intermed} can be written
\begin{equation}\label{eq:finalT}
  \int_{-T}^{T}
  \|v\|_X^2 dt\lesssim
  \| v(T)\|_{L^2}^2+\|v(-T)\|_{L^2}^2=
  2\|g\|_{L^{2}}
\end{equation}
and letting $T\to \infty$ we conclude that
\begin{equation*}
  \int_{-\infty}^{\infty}
  \|e^{it\HH}g\|_X^2 dt\lesssim\|g\|_{L^{2}}
\end{equation*}
which is exactly \eqref{eq:smoodirac} for $g\in{\rm Ran}(D_A)$, which is dense in $L^2$. So density arguments provide \eqref{eq:smoodirac}.

In order to prove \eqref{eq:smoodirac2}, let us come back to \eqref{eq:LHS}. If we take the strict inequality in \eqref{eq:algebra}, which is equivalent to assume the strict inequality in \eqref{eq:condition}, for instance the minorant being some positiv $\epsilon$, we have that
\begin{align}\label{eq:LHS2}
&2\int_{t=-T}^{t=T}\int_{\R^3}\nabla_AuD^2\phi_R\overline{\nabla_A u}dt
    -\frac12\int_{t=-T}^{t=T}\int_{\R^3}|u|^2\Delta^2\phi_Rdt
    \\
  & +2\Im\int_{t=-T}^{t=T}\int_{\R^3}u\phi'_RB_{\tau}\cdot\overline{\nabla_A u}
    +2\int_{t=-T}^{t=T}\int_{\R^3}|u|^2S\cdot\left[\nabla\phi_RDB\right]dt
    \nonumber
    \\
   & \geq \frac23\|\nabla_Au\|_{X,T}^2+\epsilon\|u\|_{Y,T}^2+\epsilon\int_{t=-T}^{t=T}\||x|^{-\frac12}\nabla_A^\tau u\|_{L^2}^2dt+2\pi\int_{t=-T}^{t=T}|u(t,0)|^2dt.
    \nonumber
\end{align}
By this inequality and \eqref{eq:boundary} we obtain
\begin{multline*} 
  \int_{-T}^{T}
  \left[
  \frac23\|\nabla_Au\|_X^2
  +\epsilon\|u\|_Y^2
  +\epsilon\||x|^{-\frac12}\nabla_A^\tau u\|_{L^2}^2
  +2\pi|u(t,0)|^2
  \right]dt 
\\
  \le
  \|\mathcal D_Au(T)\|_{L^2}^2+\|\mathcal D_Au(-T)\|_{L^2}^2.
\end{multline*}
The right hand side can be estimated using the
obvious inequality
\begin{equation*}
  \|\mathcal D_A f\|_{L^2}^2\leq\|\HH f\|_{L^2}^2,
\end{equation*}
and the conservation of $\|\HH u(t)\|_{L^2}$.
In order to complete
the proof of \eqref{eq:smoodirac2}, it only remains to remark that
the term $\|u\|_{L^\infty_xL^2_t}$ in the inequality
is obtained by the term $\int_{-T}^{T}|u(t,0)|^{2}dt$
by translating in space the multiplier $\phi$.
Letting $T\to \infty$ we conclude the proof.

\section{Proof of Strichartz estimates}

We pass to the proof of Theorem \ref{thm:stri} for
the massless case $\HH=\mathcal{D}_{A}$. We rewrite
$u=e^{it\mathcal D_A}f$ using the Duhamel formula:
\begin{equation}\label{eq:duhamel}
  u(t)=
  e^{it\mathcal D}f+\int_0^te^{i(t-s)\mathcal D}\alpha\cdot Au(s)ds.
\end{equation}
The term $e^{it\mathcal D}f$ is estimated directly via \eqref{eq:strhom}.
For the Duhamel term, we follow the Keel-Tao method (see \cite{kt},
\cite{gv}):
by the Christ-Kiselev Lemma in \cite{ck}, it is sufficient to estimate the
untruncated integral
\begin{equation*}
  \int e^{i(t-s)\mathcal D}\alpha\cdot Au(s)ds=
  e^{it\mathcal D}\int e^{-is\mathcal D}\alpha\cdot Au(s)d
\end{equation*}
since we are only interested in the non-endpoint case.
Again by \eqref{eq:strhom} we obtain
\begin{equation}\label{eq:T}
  \left\|e^{it\mathcal D}\int_0^te^{-is\mathcal D}\alpha\cdot Au(s)ds
  \right\|_{L^p\dot H^{\frac1q-\frac1p-\frac12}_q}\lesssim
  \left\|\int e^{-is\mathcal D}\alpha\cdot Au(s)ds\right\|_{L^2}.
\end{equation}
Now we use the dual form of the
smoothing estimate \eqref{eq:smoodirac}, i.e.
\begin{equation}\label{eq:Tstar}
  \left\|\int e^{-is\mathcal D}\alpha\cdot Au(s)ds\right\|_{L^2}
  \leq
  \sum_{j\in\Z}2^{\frac j2}\||A|\cdot|u|\|_{L^2_tL^2(|x|\sim2^j)}.
\end{equation}
Hence, by H\"older inequality, assumption \eqref{eq:decay} and estimate \eqref{eq:smoodirac} we continue the estimate as follows
\begin{equation}\label{eq:final}
  \sum_{j\in\Z}2^{\frac j2}\||A|\cdot|u|\|_{L^2_tL^2(|x|\sim2^j)}
  \leq
  \sum_{j\in\Z}2^j\sup_{|x|\sim2^j}|A|\cdot \sup_{j\in\Z}2^{-\frac j2}\|u\|_{L^2_tL^2(|x|\sim2^j)}
  \lesssim\|f\|_{L^2},
\end{equation}
and this concludes the proof of \eqref{eq:stridir2}.

We pass now to the proof of \eqref{eq:stridir3} in the massive case.
By mixing free Strichartz estimates with the dual of \eqref{eq:smoodirac},
for the Duhamel term we obtain
\begin{equation}\label{eq:TTstar2}
  \left\|e^{it\mathcal H_0}\int_0^te^{-is\mathcal H_0}\alpha\cdot Au(s)ds
  \right\|_{L^p H^{\frac1q-\frac1p-\frac12}_q}\lesssim
  \sum_{j\in\Z}2^{\frac j2}\||A|\cdot|u|\|_{L^2_tL^2(|x|\sim2^j)},
\end{equation}
for any Schr\"odinger admissible couple $(p,q)$, with $p\geq2$. The endpoint here can be recovered by using exactly
the same technique as in \cite{IK}, Lemma 3.
The rest of the proof is completely analogous to the massless case.

\appendix

\section{Magnetic Hardy Inequality for Dirac}\label{sec:hardy}

We now prove Proposition \ref{thm:hardy}.
Denote by $(\cdot,\cdot)$ the inner product in $ L^2(\R^3,\C^4)$, $\|\cdot\|$ the associated norm, and observe
that, due to the formula \eqref{eq:spin}, we have the relation
\begin{align*}
  \|\mathcal D_A f\|^2 & =(\alpha\cdot\nabla_Af,\alpha\cdot\nabla_Af)
  =-((\alpha\cdot\nabla_A)(\alpha\cdot\nabla_A)f,f)
  \\
  &
  =-(\nabla_A^2f,f)
  -2i(S\cdot(\nabla_A\wedge\nabla_A)f,f),
\end{align*}
where $S=\frac i4\alpha\wedge\alpha$ is the spin operator.
Writing for brevity $\partial_j^A=\partial_j-iA^j$,
we can compute explicitly
\begin{equation*}
   \nabla_A\wedge\nabla_A  =\left([\partial_2^A,\partial_3^A],
  [\partial_3^A,\partial_1^A],
  [\partial_1^A,\partial_2^A]\right)=iB,
\end{equation*}
where $B=\text{curl}A$. Hence by the previous relation we obtain
\begin{equation*}
  0\leq\|\mathcal D_A f\|^2=\|\nabla_Af\|^2
  +2(S\cdot Bf,f).
\end{equation*}
Notice that $S$ is a triple of matrices of norm $\le1/2$, hence
we can write
\begin{equation*}
  |2(S\cdot Bf,f)|\le C_{0}\||x|^{-1}f\|^{2}
\end{equation*}
and this implies
\begin{equation}\label{eq:diraccurl}
  \|\mathcal D_A f\|^2\ge\|\nabla_Af\|^2-
  C_{0} \|\frac{f}{|x|}\|^2-\|B_{2}\|_{L^{\infty}}\|f\|^{2}.
\end{equation}
Now we recall the magnetic Hardy inequality
\begin{equation}\label{eq:ardimagn}
  \frac14\int\frac{|f|^2}{|x|^2}dx
  \leq\int|\nabla_Af|^2dx,
\end{equation}
which is proved in \cite{FV}.
To complete the proof of \eqref{eq:hardy}, it is sufficient to write
\begin{equation*}
  \|\HH f\|_{L^{2}}^{2} =(\HH^{2}f,f)=
  m^{2}\|f\|^{2}+\|\mathcal{D}_{A}f\|^{2}
\end{equation*}
and use the preceding estimates.

The $\epsilon$ inequality is obtained using
\begin{equation*}
  (1-\epsilon)\frac14\int\frac{|f|^2}{|x|^2}dx+\epsilon \int|\nabla_Af|^2dx
  \leq\int|\nabla_Af|^2dx.
\end{equation*}

\end{document}